\documentclass[12pt]{amsart}
\usepackage{amsmath}
\usepackage{amssymb}
\usepackage{amsfonts}
\usepackage{amsthm}
\usepackage{verbatim}
\usepackage{amscd}
\usepackage{cite}
\usepackage{leftidx}
\usepackage{enumerate}
\usepackage{txfonts}
\usepackage{manfnt}
\usepackage{amscd}
\usepackage[mathscr]{eucal}
\usepackage{hyperref}
\usepackage{graphicx}
\usepackage{palatino}

\def\tD{{\tilde D}}

\DeclareMathOperator{\del}{\partial}

\DeclareMathOperator{\dlog}{dlog}

\def\refp #1.{(\ref{#1})}
\newcommand\carets [1]{\langle #1 \rangle}

\def\sbr #1.{^{[#1]}}
\def\sfl #1.{^{\lfloor #1\rfloor}}

\def\what{\widehat}

\def\?{{\bf{??}}}

\def\dgla{differential graded Lie algebra\ }

\def\HH{\mathbb H}

\def\C{\mathbb C}
\def\P{\mathbb P}

\def\sym{\text{\rm Sym} }

\def\pf{\mathrm {Pf}}

\def\O{\mathcal O}

\def\g{\mathfrak g}

\def\1/2{\frac{1}{2}}

\def\I{\mathcal{ I}}

\def\d{\mathrm{d}}

\def\2{{[2]}}

\def\nl{\newline}

\def\<{\langle}
\def\>{\rangle}

\def\2{{[2]}}

\def\Def{\mathrm{Def}}

\def\scl #1.{^{\lceil#1\rceil}}
\def\spr #1.{^{(#1)}}
\def\sbc #1.{^{\{#1\}}}

\def\subpr#1.{_{(#1)}}
\def\ds{\vskip 1cm}

\def\beq{\begin{equation*}}
\def\eeq{\end{equation*}}

\newcommand{\sing}{{\mathrm{sing}}}

\newcommand{\llog}[1]{\langle\log {#1} \rangle}
\newcommand{\mlog}[1]{\langle-\log {#1} \rangle}
\def\g3{{\Gamma\spr 3.}}

\def\Def{\mathrm {Def}}

\newcommand{\eqspl}[2]{
\begin{equation}\label{#1}
\begin{split}
#2\end{split}\end{equation}}

\newcommand{\beginalphaenum}{
\begin{enumerate}\renewcommand{\labelenumi}{ }
\item \begin{enumerate}
}

\def\eex{\end{rm}\end{example}}

\newtheorem*{lem*}{Lemma}

\newtheorem{thm}{Theorem} 

\newtheorem*{thm*}{Theorem}
\newtheorem*{prop*}{Proposition}
\newtheorem{cor}[thm]{Corollary}
\newtheorem*{cor*}{Corollary}

\newtheorem{lem}[thm]{Lemma}

\newtheorem*{claim*}{Claim}
\newtheorem{prop}[thm]{Proposition}

\newtheorem{defn}[thm]{Definition}
\theoremstyle{remark}

\newtheorem{rem}[thm]{Remark}
\newtheorem{crit-rem}[thm]{Critical remark}
\newtheorem{remarks}[thm]{Remarks}
\newtheorem*{remarks*}{Remarks}

\newtheorem{example}[thm]{Example}
\newtheorem*{example*}{Example}

\newtheorem*{defn*}{Definition}

\begin{document} 
\title{Deformations of holomorphic \\ pseudo-symplectic Poisson manifolds}
\author 
{Ziv Ran}



\thanks{arxiv.org/1308.2442}
\date {\today}


\address {\nl UC Math Dept. \nl
Big Springs Road Surge Facility
\nl
Riverside CA 92521 US\nl 
ziv.ran @  ucr.edu\nl
\url{http://math.ucr.edu/~ziv/}
}


\begin{abstract}
We prove unobstructed deformations for
 compact K\"ahlerian even-dimensional Poisson manifolds whose Poisson tensor degenerates
along a divisor with mild singularities. Examples include Hilbert schemes of del Pezzo surfaces.

\end{abstract}
\maketitle
By definition, a \emph{pseudo-symplectic}
\footnote{sometimes called log-symplectic in the literature}
 manifold is a compact complex
manifold $X$ satisfying the $\del\bar\del$ lemma,
 endowed with a (holomorphic) Poisson structure $\Pi\in H^0(\bigwedge\limits^2 T_X)$
which is generically of full rank $2n:=\dim(X)$. This implies that 
 $\Pi$ degenerates along a (possibly trivial) anticanonical
Pfaffian divisor \[D=\pf(\Pi)=[\Pi^{n}]\in |-K_X|.\]
 $D$ is trivial precisely when $\Pi$ is
a symplectic Poisson structure. We will denote by $\pf^k(\Pi)$ the locus where
$\Pi$ has corank $2k$ or more; in terms of a skew-symmetric matrix
representation of $\Pi$, $\pf^k(\Pi)$ is the locus defined by the $(2k)\times(2k)$
minors based on the same set of rows as columns.\par
Our interest in this paper is in deformations of pseudo-symplectic Poisson manifolds whose Pfaffian
divisor has mild singularities. A strong form of mildness is the following.
\begin{defn}A pseudo-symplectic Poisson manifold
$(X, \Pi)$ is said to be 
\emph{Pfaffian-normal} or \emph{P-normal} if
$D=\pf(\Pi)$ is reduced with (local) normal crossings, and
$D$ is smooth at every point where $\Pi$ has corank exactly 2.
\end{defn}
\begin{remarks*}
(i) The P-normal hypothesis
means, besides normal crossings for $D$, that $\Pi$ has corank $>2$ at all singular points of $D$.
Thus, for example, the Poisson structure $x_1x_2\del_{x_1}\wedge\del_{x_2}$ on the plane is not P-normal,
though its Pfaffian has normal crossings. On the other hand, the product of smooth surfaces
each endowed with a smooth anticanonical divisor admits a natural P-normal Poisson structure.
\par
(ii) We will prove below (Proposition \ref{normal-form-prop})
that P-normal Poisson structures admit the following local normal form
\eqspl{normal-form-eq}{
x_1\del_{x_1}\wedge\del_{y_1}+...+x_k\del_{x_k}\wedge\del_{y_k}+
\del_{x_{k+1}}\wedge\del_{y_{k+1}}+...+\del_{x_n}\wedge\del_{y_n}.
}
Here $k$ is the miltiplicity of $D$ at the origin, which coincides with
half the corank.

\end{remarks*}
\par P-normal Poisson structures other than products seem a bit hard to find 'in nature', so we make the following more general
definition, whose import is that the Pfaffian divisor has, in a strong sense, 'generic' singularities.
\begin{defn}
A Poisson structure  $(X, \Pi)$ is said to  be \emph{weakly P-normal } if  for every $k\geq 1$\par
(a) $\pf^k(\Pi)$ has codimension $k$ and $\pf^k(\Pi)-\pf^{k+1}(\Pi)$ is smooth;\par
(b) generically along $\pf^k(\Pi) $, $\pf^{k-1}(\Pi)$ has $k$ smooth pairwise transverse
branches;\par
(c) in the iterated blowup of $X$ along the flag $\pf^n(\Pi)\subset \pf^{n-1}(\Pi)\subset
...\subset\pf^2(\Pi)$, the reduced total transform of $\pf(\Pi)$ is a divisor with normal crossings.
\end{defn}

\begin{remarks*}

(i) As a consequence of the above normal form, we will show below (Corollary \ref{p-wp-cor}) 
that P-normality implies weak P-normality.
\par
(ii) The notion of weak P-normality is
motivated the following example. Any anticanonical curve $C$ on a smooth surface $S$
is the degeneracy divisor of a Poisson structure $\Pi$ on $S$.
By a construction of Bottacin \cite{bottacin} (see below, \S\ref{hilb}), 
$\Pi$ induces a Poisson structure
$\Pi\sbr r.$ on the punctual Hilbert scheme $S\sbr r.$.
It is proven in \S \ref{hilb} for $r=2$ and in \cite{incidence} for $r\geq 2$
that whenever $C$ is smooth, $\Pi\sbr r.$ is weakly P-normal (but not P-normal).
This yields a large class of weakly P-normal Poisson structures on, e.g. Hilbert schemes of
Del Pezzo surfaces, and our unobtructedness results apply to those.
\end{remarks*}
If $\pf(\Pi)$ is trivial, then $\Pi$ is a symplectic Poisson structure, 
so $X$ is a Calabi-Yau ($K$-trivial) manifold, hence by
a well-known result of Tian-Todorov has unobstructed deformations. 
At the other extreme, if $\pf(\Pi)$ is ample, then $X$ is Fano, hence trivially
has unobstructed deformations (a less trivial unobdstructedness result for 
\emph{weak} Fano manifolds was recently obtained by T. Sano \cite{sano}).
If $D=\pf(\Pi)$ is smooth, 
then a result of \cite{iacono} (sketched
in the Appendix, \S\ref{appendix} below, assuming only $D$ has normal crossings) 
shows that
the pair $(X, D)$ has unobstructed locally trivial deformations.
Goto \cite{goto} and Hitchin \cite{hitchin-poisson} have obtained some
results on unobstructedness of certain 'special' (and 1-parameter)
Poisson deformations.
See also \cite{fiorenza-manetti}.
Our purpose
here is to sharpen these results in the weakly P-normal case, by proving that a for 
weakly P-normal
Poisson manifold $(X, \Pi)$, all (Poisson) deformations are
unobstructed, 
i.e. the full Poisson deformation space $\mathrm{Def}(X, \Pi)$ is smooth:
\begin{thm}\label{main-thm}
Let $(X, \Pi)$ be a weakly P-normal Poisson manifold with Pfaffian divisor $D$. Then\par
(i)  $(X,\Pi)$ has unobstructed Poisson deformations.\par
(ii) Poisson deformations of $(X, \Pi)$ induce locally trivial deformations on $D$.
\par 
(iii) There is a space of deformations of the triple $(X, \Pi, D)$,
 whose
forgetful morphisms to $\mathrm{Def}(X, \Pi)$ and $ \mathrm{Def}_{\mathrm{loc.\ trivial}}(X,D)$
are both smooth.

\end{thm}
We do not prove that the natural morphism $\mathrm{Def}(X, \Pi)\to \mathrm{Def}_{\mathrm{loc.\ trivial}}(X,D)$ is smooth. However, the fact that every locally trivial deformation of $(X, D)$ lifts
to a deformation of $(X, \Pi, D)$ is sufficient to prove:
\begin{cor}
Given a weakly P-normal Poisson manifold $(X,  \Pi)$ with 
Pfaffian $D$ and a deformation $\tilde X$
of $X$, $\Pi$ extends to $\tilde X$ iff $D$ extends to $\tilde X$ locally trivially.
\end{cor}
As noted above, these results apply to the Hilbert Poisson manifolds $(S\sbr r., \Pi\sbr r.)$
whenever the curve $[\Pi]$ is smooth. To put matters in perspective
however, there exist even Poisson \emph{surfaces} with obstructed deformations,
as the following example shows. In this example, the Pfaffian divisor $[\Pi]$
is non-reduced with normal crossings support; we don't know an example 
of an obstructed Poisson surface with reduced, normal-crossings Pfaffian. 
\footnote{A statement in \cite{katzarkov-kontsevich-pantev}
(Lemma 4.19) seems to imply no such example exists} \par
\begin{example*}
Let $L$ be a line in $\P^2$, $p_1,...,p_4$ general points on $L$,
and $p_5, p_6, p_7$ general points in $\P^2$. Let
$Z=(2p_1+2p_2)_L+p_3+...+p_7$. Let
$C$ be the unique conic through $p_1, p_2, p_5,p_6,p_7$.
Let $\pi:X\to\P^2$ be the blowup of $Z$, with exceptional divisor $E$
(blowing up $2p_L$ means blow up $p$ then blow up
the appropriate point on the exceptional line).
Let $D=\pi^*(L+C)-E$, an anticanonical divisor on $X$
with normal crossing support (but 2 components of multiplicity 2),
which corresponds to a Poisson structure $\Pi$.\par
The pair $(X, D)$ can be deformed (generalized) in 2 ways:\par
1. Deform $(2p_i)_L$ to a general length-2 scheme supported at $p_i$,
$i=1,2$,
which is still a subscheme of $L+C$.\par
2. Deform $C$ to a general conic through $p_5, p_6, p_7$ (not going
through $p_1, p_2$).\par
Thus, the deformation space of $(X, D)$, or equivalently of $(X, \Pi)$,
is locally reducible, hence obstructed.\qed
\par
\includegraphics[scale=.25]{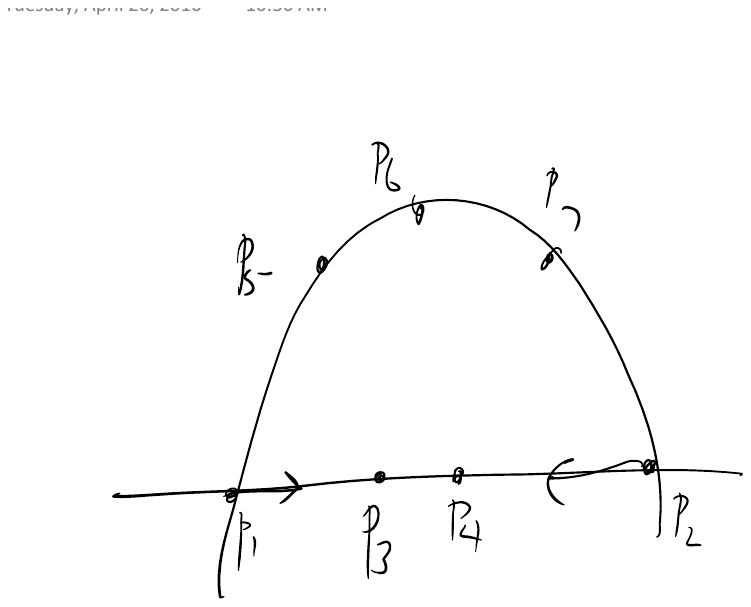}
\end{example*}
%
%

The strategy of the proof of the Theorem in the P-normal case
is analogous to that in the symplectic case: namely, relate infinitesimal
symmetries (in this case, a certain 'log Schouten dgla') to differentials (in this case,
a certain log complex), and apply Hodge theory (in this case, Deligne's $E_1$
degeneration theorem for the log complex). The strategy in the weakly P-normal case
is to relate deformations of $(X, \Pi)$ to those of a suitable P-normal blowup.\par
\subsection*{Acknowledgments} I thank Chunghoon Kim for many
instructive discussions about Poisson manifolds. I thank the referee
for his helpful comments and  especially for providing some missing references. 
\section{Basics and examples of Poisson structures}
See \cite{dufour-zung} or \cite{ciccoli} for basic facts on 
Poisson manifolds and \cite{ginzburg-kaledin} (especially the appendix)
or \cite{namikawa}, and references therein
 for information on deformations of
Poisson complex structures. For a different, more elementary 
construction of deformations,
see the Riverside thesis of C. Kim (to appear).
\subsection{Schouten algebra} See \cite{brylinski}.
Fixing a manifold $X$, we denote by $T=T_X$ its tangent Lie algebra of holomorphic vector
fields. Then the exterior algebra $T^.=\bigoplus \bigwedge\limits^iT$ is endowed
with a graded bracket known as the Schouten bracket, extending the Lie bracket,
with the property that
\eqspl{schouten}{
[P, Q\wedge R]=[P,Q]+(-1)^{(p-1)q}Q\wedge[P, R], P\in T^p, Q\in T^q.
} A Poisson structure on $X$ is by definition a tangent bivector field $\Pi\in H^0(\bigwedge\limits^2 T)$
satisfying the integrability condition  $[\Pi, \Pi]=0$.
A Poisson structure $\Pi$ defines on $T^.$ the structure of a complex with differential
$[.,\Pi]=L_\Pi(.), L=$ Lie derivative, 
hence a structure of \emph{\dgla} (dgla) called the \emph{Schouten algebra}, 
which is just the algebra
of infinitesimal symmetries of the pair (complex structure, Poisson structure). 
The deformation theory associated to
the Schouten dgla  is likewise the deformation theory of the Poisson structure $\Pi$
together with the underlying complex structure.\par
The Schouten dgla $T^. $ admits as dgla module the 'reverse de Rham' dg complex
$(\Omega^., \delta)$, with codiferential
\eqspl{codifferential}{\delta=L_\Pi=i_\Pi d-d i_\Pi,}
of  degree -1 and square zero thanks to $[\Pi, \Pi]=0$.
This differential has the following properties, which together determine it uniquely:
\eqspl{co-derivative-eq}{
(i) \ \delta(f\omega)&=f\delta(\omega)+i_{\nabla^\Pi(f)}\omega, \nabla^\Pi(f)=i_{df}\Pi,\\
(ii) \ \delta(df)&=0}
These imply
\eqspl{co-derivative-2-eq}{
(i) [d, \delta]=0,\\
(ii) \  \delta(df_1\wedge df_2)&=\delta(f_1df_2)=i_{df_1\wedge df_2}\Pi,\\
(iii) \  \delta(df_1\wedge...\wedge df_k)&=\sum (-1)^{i+j}\delta(df_i\wedge df_j)\wedge df_1\wedge...\what{df_i}...\what{df_j}...\wedge df_k
} Note that in view of these properties, the restriction $\delta\spr 1.$ of $\delta$ on $\Omega^1$ 
determines $\delta$ on the
rest of the complex. Also, any bivector field $\Pi$ allows us to define maps $\delta\spr 1.,...,\delta\spr n.$,
but it is the integrability condition $[\Pi,\Pi]=0$ - which is equivalent to
 $\delta\spr 2.\delta\spr 1.=0$ - that makes it a complex.\par
Poisson deformations are $T^.$-deformations of this complex, i.e. deformations $\tilde X/S$ of
the complex structure of $X$, together with deformations $\tilde\delta$ 
of the codifferential $\delta $ on the relative de Rham
complex $\Omega^._{\tilde X/S}$ satisfying \eqref{co-derivative-eq},\eqref{co-derivative-2-eq}, 
which
simply mean that $\tilde\delta$ comes from a relative tangent bivector field
 $\tilde\Pi\in H^0(\bigwedge\limits^2 T_{\tilde X/S})$ on
$\tilde X/S$, plus the square-zero condition, which is equivalent to
 $\tilde\delta\spr 1.\tilde\delta\spr 2.=0$, and then again to $[\tilde\Pi, \tilde\Pi]=0$.
 The dgla homomorphism $T^.\to T$ induces a morphism of formal deformation
 spaces \[\mathrm{Def}(X,\Pi)\to \mathrm{Def}(X).\]
In general,  this morphism is clearly not injective (and probably not surjective).
\subsection{An identity}
For future reference, we note an identity. Let $\Phi$ be a closed holomorphic  form, defined locally. Then,
where $L$ denotes Lie derivative, we have by definition,
\[L_\Pi\carets{\Pi, \Phi}=d\carets{\Pi^2, \Phi}-\carets{\Pi, d\carets{\Pi, \Phi}}.\]
On the other hand by the derivation property of Lie derivatives and
the fact that $L_\Pi\Pi=[\Pi, \Pi]=0$, we have
\[L_\Pi\carets{\Pi, \Phi}=\carets{\Pi, L_\Pi\Phi}=\carets{\Pi, d\carets{\Pi, \Phi}}.\]
Consequently,
\[ d\langle \Pi^2, \Phi\rangle=2\langle\Pi, d\langle\Pi, \Phi\rangle\rangle\]
and inductively,
\[d\carets{\Pi^k, \Phi}=k\carets{\Pi^{k-1}, d\carets{\Pi, \Phi}}.\]
Consequently, we have
\eqspl{pfaffian-identity}{
\carets{\Pi, d\carets{\Pi^n,\Phi}}=n\carets{\Pi, \carets{\Pi^{n-1}, \carets{d\carets{\Pi, \Phi}}}}
=n\carets{\Pi^n, d\carets{\Pi, \Phi}}.
}
Now suppose $X$ is  a pseudo-symplectic Poisson manifold of dimension $2n$,
and $\Phi$ is a local holomorhic volume form. Then
$\carets{\Pi^n, \Phi}$ is a local equation for the Paffian divisor $D$, so at a smooth point of $D$,
$d\carets{\Pi^n, \Phi}$ generates the conormal sheaf. So the latter identity shows that
the conormal sheaf to $D$ is contained in the kernel of $\Pi$ along $D$ (which has rank 2 or more),
thus:
\begin{lem}\label{conormal-in-kernel-lem}
If $(X, \Pi)$ is pseudo-symplectic with Pfaffian divisor $D$, then at a smooth point of $D$ the conormal line
is contained in the kernel of $\Pi$.
\end{lem}
\subsection{Duality}
$\Pi$ gives rise to various duality or interior multiplication maps, such as 
\[\Pi^\#:\Omega\to T,\]
\[i_\Pi:\Omega^p\to\Omega^{p-2}, p\geq 2.\]
\subsection{Poisson- de Rham algebra}
A fundamental feature of Poisson geometry is that the cotangent exterior algebra $\Omega_X^.$
is endowed with a graded Lie-Poisson bracket with the properties
\eqspl{}{
[df, dg]=d\{f, g\},\\
[\omega_1, f\omega_2]=f[\omega_1, \omega_2]+\Pi(\omega_1\wedge df)\omega_2,
} plus a derivation property analogous to \eqref{schouten}.\par
The standard duality map
\[\Pi^\#:\Omega\to T\]
extends to
\eqspl{}{
\bigwedge\Pi^\#:(\Omega^., \wedge, [,])\to (T^., \wedge, [,])
} as a graded homomorphism with respect to both $\wedge$ and $[,]$
products.
\subsection{Products}\label{products}
Note that if $(X_i, \Pi_i)$ are Poisson manifolds, then $X=\prod\limits_i X_i$
admits the Poisson structure $\Pi=\sum\limits_i p_{X_i}^*(\Pi_i)$, called the product
structure  of the $(X_i, \Pi_i)$. Moreover, if each $(X_i, \Pi_i)$ is P-normal, then
so is $(X, \Pi)$.\par
Note that an effective anticanonical divisor $D$ on a smooth surface $X$ 
determines a Poisson structure on $X$,
which is P-normal if $D$ is smooth. Taking products of these yields examples
of P-normal Poisson manifolds with nonsmooth Pfaffian divisor.
\subsection{Hilbert schemes: Poisson au  Beauville}\label{hilb}
Beauville \cite{bvl} has constructed a symplectic structure on the Hilbert scheme
of a symplectic (i.e. K3 or abelian) surface. This construction was extended
by Bottacin \cite{bottacin} to yield
a Poisson structure $\Pi\sbr r.$ on the Hilbert scheme $S\sbr r.$ of a Poisson (i.e. anticanonical) 
surface $(S, \Pi)$- an object in much greater supply than its symplectic equivalent.
This structure is not P-normal, but as proven in \cite{incidence}, it is weakly P-normal, provided
the corresponding anticanonical curve $\pf(\Pi)$ is smooth.
Here we will give an explicit construction of $\Pi\sbr r., r\geq 2$,
and a proof of weak P-normality in the initial case
$r=2$.\par
Thus, let $S$ be a smooth surface endowed with a Poisson structure $\Pi$ or
equivalently, an effective anticanonical divisor $C=\mathrm{Pf}(\Pi)$.
Let $\tilde S^2$ denote the blowup of $S^2=S\times S$ in the diagonal, with exceptional divisor $E_2
=\P(\Omega^1_S)$. Then we have a diagram
\eqspl{}{
\begin{matrix}
&&\tilde S^2&&\\
&p\swarrow&&\searrow q&\\
S\sbr 2.&&&&S^2
\end{matrix}
} with both $p$ and $q$ (simply) ramified along $E_2$,
which goes isomorphically via $p$ to the discriminant $\Delta_s$.

Now the product Poisson structure $\Pi^2$ on $S^2$ pulls back to a meromorphic
Poisson structure $\tilde\Pi_2$ on  $\tilde S^2$, invariant under the flip involution
and
with a simple pole along $E_2$ ($\Pi_2$ degenerates along $C\times S\cup S\times C$).
Then $\tilde\Pi_2$ descends via $p$ to a 
meromorphic Poisson structure $\Pi\sbr 2.$ on $S\sbr 2.$ with
at worst a pole on the discriminant $\Delta_2$. We claim $\Pi\sbr 2.$ has in
fact no pole on $\Delta_2$, hence is holomorphic. This can be checked at
a generic point of $\Delta_2$, as follows.\par
Now locally near a generic point of $S$ we can write $\Pi=\langle V, \Phi\rangle$
where $\Phi$ is a symplectic form and $V$ is a covolume (dual volume) form. Up on the blowup
$\tilde S^2$, we get for the respective pullbacks
\[\tilde\Pi^2=\langle\tilde V^2,\tilde\Phi^2\rangle\]
Now $\tilde V^2$ has a simple pole along $E_2$, hence is the pullback of
a (regular, everywhere nondegenerate)
covolume form $V\sbr 2.$ on $S\sbr 2.$. By Beauville, $\tilde\Phi^2$ is
the pullback of a symplectic form $\Phi\sbr 2.$ on $S\sbr 2.$. Consequently,
\[\tilde\Pi^2=p^*(\Pi\spr 2.), \Pi\sbr 2.=\langle V\sbr 2., \Phi\sbr 2.\rangle\]
and $\Pi\sbr 2.$ is holomorphic, as claimed.\par
 Note that by the same argument, $\Pi\sbr 2.$
is generically nondegenerate on $\Delta_2$, so its Pfaffian $D=\pf(\Pi\sbr 2.)$
coincides
with the locus $S+C$ of schemes meeting $C$, and $D$ clearly has normal crossings
away from the locus $2C$ of 'double points on $C$', i.e. ideals of the
form $\I_p^2+\I_C, p\in C$.
\par
Let $\pi:X_2\to S\sbr 2.$ denote the blowup of $S\sbr 2.$ in the singular locus of $D$, i.e. $C\spr 2.$.
Let $E$ be the exceptional divisor and $\tilde D$ the proper transform of $D$. Then
we have
\[\pi^*(D)=\tilde D+2E, K_{X_2}=\pi^*(K_{S\sbr 2.})+E,\]
therefore
\[-K_{X_2}=\tilde D+E.\]
As above, the Poisson structure $\Pi\sbr 2.$ on $S\sbr 2.$ induces a Poisson structure
$\Pi_{X_2}$ on $X_2$
with Pfaffian divisor $\tilde D+E$. The following simple local calculation shows
that $\tilde D+E$ has normal crossings.\par

Let $x, y$ be local coordinates on $S$ so that $x$ defines $C$, and consider
a subscheme of the form $(x, y^2)$, a typical non-normal-crossing point of $D$.
Deformations of this scheme on $S$ are of the form $(x+a_0+a_1y, y^2+b_0+b_1y)$
so $(a_0, a_1, b_0, b_1)$ are local coordinates on $S\sbr 2.$. Then the equation of $D$,
the locus of schemes meeting $C$ nontrivially, is $a_0^2+b_0a_1^2 -b_1a_0a_1$.
The double locus $C\spr 2.$ is given by the equations $a_0, a_1$. Blowing this up gives,
on the relevant open,
\[a_1^2([a_0/a_1]^2+b_0+b_1[a_0/a_1])\]
where $a_1$ is the equation of $E$ and $[a_0/a_1]^2+b_0+b_1[a_0/a_1]$ 
is the equation of $\tilde D$. These are clearly transverse, so $\tilde D+E$
has normal crossings.\par
Now, it is easy to see that at a point of $C\sbr 2.\setminus \Delta_2$, $\Pi\sbr 2.$
is locally of the form
\[x_1\del_{x_1}\wedge\del_{y_1}+x_2\del_{x_2}\wedge\del_{y_2}\]
where $x_1, x_2$ are equations for $C\spr 2.$. This proves that $\Pi\sbr 2.$ is weakly
P-normal. Furthermore, by an easy computation,
the logarithmic vector fields $x_1\del_{x_1}, x_2\del_{x_2}$ lift to $X$ as holomorphic
vector fields which are vertical along $E$ and vanish on $\tilde D\cap E$
which is the singular locus of $\pf(\Pi_X)$. This implies that $\Pi_X$ 
vanishes along the singular locus of $\pf(\Pi_X)$, so $\Pi_{X_2}$ is P-normal.\par
In fact, by an elementary but lengthy calculation (see \S \ref{computation}), 
one can compute $\Pi\sbr 2.$ explicitly:
in the above coordinates, we have
\eqspl{2Pi}{
2\Pi\sbr 2.=2(b_0a_1-b_1a_0)\del_{a_0}\del_{b_0}-2a_0\del_{a_1}\del_{b_0}-a_1^2\del_{a_0}\del_{a_1}
-a_0\del_{a_0}\del_{b_1}-a_1\del_{a_1}\del_{b_1}
}
which shows that $\Pi\sbr 2.$ vanishes precisely on the locus $C\spr 2.$, with ideal $(a_0, a_1)$
and also shows again that the degeneracy scheme of $\Pi\sbr 2.$ is $D$ with equation
$a_0^2+b_0a_1^2 -b_1a_0a_1$
.\par

\par
In the general case $r\geq 2$, assuming $\Pi\sbr r.$ exists, 
 we will construct $\Pi\sbr r+1.$ by induction, based on the diagram
\eqspl{}{
\begin{matrix}
&&\tilde S\sbr r+1.&&\\
&p\swarrow&&\searrow q&\\
S\sbr r+1.&&&&S\sbr r.\times S
\end{matrix}
}
where $q$ is the blowing-up of the tautological subscheme, $\tilde S\sbr r+1.$
is the $(r, r+1)$ flag-Hilbert scheme, and $p$ is the forgetful map.
Inductively, $\Pi$ induces a Poisson stricture $\Pi\sbr r.$ on $S\sbr r.$,
whence a product Poisson structure  on $S\sbr r.\times S$
and thence a meromorphic Poisson
structure on $\tilde S\sbr r+1.$, which again descends to a Poisson
structure $\Pi\sbr r+1.$ on $S\sbr r+1.$, possibly with a pole along the
discriminant. Then a local analysis as above at a generic point of the discriminant,
which corresponds to a double point plus $r-1$ generic points of $S$,
shows that $\Pi\sbr r+1.$ is holomorphic.\par

\section{The log complexes}
\subsection{Cotangent log complex}
(see \cite{deligne-hodge-ii}, \cite{grif-schmid})
Let $X$ be a compact K\"ahler  manifold with holomorphic de Rham complex $\Omega^._X$.
Let $D$ be a reduced divisor with local normal crossings on $X$. Let $\tilde D$
be the normalization (=desingularization) of $D$. The sheaf of log
1-forms on $X$ with respect to $D$, denoted by $\Omega^1_X\langle \log D\rangle$
is by definition the kernel of the natural map
\eqspl{}{
\Omega^1_X(D)\to \Omega_\tD(D)}
In terms of a local coordinate system $x_1,...,x_n$ such  that $D$ has local
equation $x_1\cdots x_k$, $\Omega^1_X\langle\log D\rangle$ is generated by
$\d x_1/x_1,...,\d x_k/x_k, \d x_{k+1},...,\d x_n$. Therefore 
$\Omega^1_X\langle\log D\rangle$ is locally free.
Define $\Omega_X^.\langle \log D\rangle$ as the subalgebra of the meromorphic exterior
algebra $\Omega_X^.(*D)$ generated by $\Omega_X^1\llog{D}$.
Then each $\Omega^r\llog{D}$ is locally free. It is a theorem of
Deligne (\cite{deligne-hodge-iii}, (8.1.9); see also \cite{zucker},  \cite{steenbrink})
 that  the 'log Hodge-to-de Rham' spectral sequence
\eqspl{}{
E_1^{p,q}=H^q(X, \Omega_X^p\llog{D})\Rightarrow \mathbb H^{p+q}(X, \Omega^.\llog{D})
}
degenerates at $E_1$. By another theorem of Deligne (\cite{deligne-hodge-ii}; see also
\cite{grif-schmid}), we have
\eqspl{}{
\mathbb H^{i}(X, \Omega^.\llog{D})\simeq H^{i}(X\setminus D, \C).
}
\subsection{ log Schouten algebra}
We denote by $T^r\mlog{D}$ the dual of $\Omega^r\llog{D}$,
which is a locally free, full-rank subsheaf of $T^r$.
It is easy to check that the Schouten bracket extends to define a graded Lie algebra structure on 
\eqspl{}{T^.\mlog{D}=\bigoplus\limits_{r\geq 1} T^r\mlog{D}.
}
Now suppose that $\Pi$ is a pseudo-symplectic Poisson structure
with Pfaffian divisor $D$. Then it is elementary that for any local
branch $D_i$ of $D$, the codifferential $\delta$ (cf. \eqref{codifferential})
vanishes locally on the subcomplex of  $\Omega^._X$ generated by $\O(-D_i)$, and therefore
descends to a codifferential on $\Omega_{\tilde D}$ so that restriction
\[(\Omega^._X, \delta)\to (\Omega^._{\tilde D}, \delta)\]
is a map of complexes.
\subsection{Duality}
Under the assumption that $\Pi$ is a pseudo-symplectic Poisson structure
with Pfaffian divisor $D$, it is easy to check that the duality map $\bigwedge\Pi^\#$ extends to 
\eqspl{}{
\Pi^\#:(\Omega^.\llog{D}, \wedge, [,])\to (T^.\mlog{D}, \wedge, [,])
} which is a dgla homomorphism. 
\begin{lem}\label{duality-iso-lem}
If $\Pi$ is a P-normal Poisson structure, then $\Pi^\#$ is an isomorphism.
\end{lem}
\begin{proof}
Obviously $\Pi^\#$ is locally an isomorphism off $D$.
At smooth points of $D$,
in terms of the normal form \eqref{normal-form-eq}, $\Pi^\#$ 
sends a local generator $dx_{1}/x_{1}$ to $\del_{x_{2}}$
and $dx_{2}$ to $-x_{1}\del_{x_{1}}$ while
\[\Pi^\#(dx_{2i-1})=\del_{x_{2i}}, \Pi^\#(dx_{2i})=-\del_{x_{2i-1}}, i=2,...,n.\]
It follows that $\Pi^\#$ is an isomorphism off the singular locus of $D$.
Because $\Pi^\#$ is a map of locally free sheaves of the same rank,
its degeneracy locus is of codimension at most 1. Therefore
$\Pi^\#$ is an isomorphism.
\end{proof}
\subsection{Normal form}
We now derive a local normal form for P-normal Poisson structures,
refining Weinstein's result. The result is elementary, but apparently not explicitly mentioned
in the literature, therefore we will give a very detailed proof.
\begin{prop}\label{normal-form-prop}
Let $\Pi$ be a P-normal Poisson structure with Pfaffian divisor $D$. Then
locally at a point where $D$ has multiplicity $k$, we can find a local coordinate system
$x_1,y_1, ...,x_{n}, y_n$ such that $D$ has equation $x_1\cdots x_k$ and $\Pi$ has the form
\eqspl{p-normal-eq}{x_1\del_{x_1}\wedge \del_{y_1}+...+x_{k}\del_{x_{k}}\wedge\del_{y_{k}}
+\del_{x_{k+1}}\wedge\del_{y_{k+1}}+...+\del_{x_{n}}\wedge\del_{y_{n}}.
}
\end{prop}
The case $k=1$, i.e. smooth points of $D$, follows from Minstein's structure theorem for Poisson structures
\cite{dufour-zung}, Thm. 1.4.5 p.14. 
\begin{cor}\label{p-wp-cor}
A P-normal Poisson structure is weakly P-normal.\qed
\end{cor}
\begin{proof}[Proof of theorem]
Working locally at a point $p$ where $\Pi$ has corank $k$, it is easy to see first of all, using Weinstein's theorem,
that we may assume $k=n$. Consider any local coordinate system
$(x_i)$ in which  $D$ has equation $x_1x_3...x_{2n-1}$. Now I claim that we may assume 
that the vector field 
\[v=\Pi^\#(dx_1/x_1)=\Pi^\#(\dlog(x_1))\] 
is nonvanishing at $p$
as plain (not logarithmic) vector field (hence, passing to a smaller
neighborhood, $v$ is nowhere vanishing). If this were false, then
\[\Pi^\#(dx_{2i-1}/x_{2i-1})(p)=0, i=1,...,n\]
(again, value as plain vector field).
On the other hand, we may replace each $x_{2i-1}$ by $x_{2i-1}\exp(x_{2i})$ to get another
local equation for $D$, and consequently
\[\Pi^\#(dx_{2i})(p)=\Pi^\#(d\log(x_{2i-1}\exp(x_{2i})))(p)-\Pi^\#(d\log(x_{2i-1}))=0.\]
This implies that every element in the image of $\Pi^\#$ vanishes at $p$, which is a
contradiction, proving our claim.\par
Now we can choose a function $y_1$ such that $v(y_1)=\langle dy_1, v\rangle=1$. 
Thus,
\[\Pi^\#\carets{dx_1/x_1\wedge dy_1}=v(y_1)=1.\]
\par
Next, I claim that by modifying each $x_{2i-1}, i>1$ by a unit factor $\exp(f_i)$, 
thus leaving $D$ unchanged, we may assume
\eqspl{perp-x}{\Pi^\#\carets{d\log(x_1)\wedge d\log(x_i)}=0, \forall i>1.}
Indeed, writing $x'_{2i-1}=x_{2i-1}\exp(f_i)$
and noting that $\Pi^\#(d\log(x_1)\wedge\omega)=\carets{v, \omega}$ for any 1-form $\omega$ (by
definition of $\Pi^\#$)
the equation  \[\Pi^\#\carets{d\log(x_1)\wedge d\log(x'_{2i-1})}=0\]
becomes the ODE in $f_i$:
\[v(f_i)=v(x_{2i-1})\]
which admits a solution $f_i$ (uniquely determined by the condition of
vanishing on the transverse hypersurface $y_1=0$).\par
Now clearly $dx_1, dx_3, ...,dx_{2n-1}$ are linearly independent sections
of the subbundle $\ker(v)\subset\Omega_X$, i.e. $x_1,x_3, ...,x_{2n-1}$
descend to part of a coordinate system on the quotient manifold of $X$
by the vector field $v$. Consequently
we may complete the functions 
$x_1, y_1, x_3,...,x_{2n-1}, $ to a local coordinate system $(x_1, y_1, x_3 y_3,...,x_{2n-1},y_{2n-1})$
on $X$ such that
\[v(y_{2i-1})=\Pi^\#\carets{d\log(x_1)\wedge dy_{2i-1}}=0, i>1.\]
For this coordinate system, we have
\[v=\del_{y_1}.\]
We let $X_0$ be the factor manifold
of $X$ by the foliation $A$ determined by $\del_{x_1}, \del_{y_1}$, which is naturally isomorphic
to each leaf of the foliation  $F$ determined by $dx_1, dy_1$  (i.e. the level sets of $(x_1, y_1)$). 
Thus we have as well
\eqspl{perp-y}{
\Pi^\#\carets{d\log(x_1)\wedge dy_{2i-1}}=v(y_{2i-1})=0, \forall i>1.
}
Note that the bivector
\[\Pi_0:=\Pi-x_1\del_{x_1}\wedge v=\Pi-x_1\del_{x_1}\wedge\del_{y_1}\]
is tangent to the leaves of the codimension-2 foliation $F$, i.e. can be written
in terms of the $\del_{x_{2i-1}}, \del_{y_{2i-1}}, i>1$ (but a priori with coefficients
depending on $x_1, y_1$). Now calculating
\[0=[\Pi,\Pi]=2[x_1\del_{x_1}\wedge\del_{y_1}, \Pi_0]+[\Pi_0, \Pi_0]\]
and considering separately the mixed terms in $x_1, y_1$ and $x_{2i-1}, y_{2i-1}, i>1$, and those
in the $x_{2i-1}, y_{2i-1}, i>1$ only, we conclude that $\Pi_0$ descends to a Poisson structure on $X_0$,
clearly P-normal with Pfaffian divisor $x_2,...,x_n$.
Applying induction, $\Pi_0$ has a normal form as desired on 
$X_0$, hence we get the desired normal form for $\Pi$.

\end{proof}
\subsection{Weak P-normality}\label{wp-sec}
\begin{prop}\label{lifting-prop}
Let $(X, \Pi)$ be a weakly P-normal Poisson manifold of dimension $2n$. Let $\pi:\tilde X\to X$
be the birational morphism obtained by blowing up $\pf^n(\Pi)$, then the proper transform of $\pf^{n-1}(\Pi)$,
etc. Then $\Pi$ lifts to a holomorphic P-normal Poisson structure on $\tilde X$.
\end{prop}
\begin{proof}
To begin with, we show that $\Pi$ lifts holomorphically, i.e. the lifted meromorphic bivector
has no poles on any component of the exceptional divisor. This can be checked locally
at a general point of a component of that exceptional divisor. Such a point
sits over a general point of some $\pf^k(\Pi)$, where $\Pi$
admits the normal form \eqref{p-normal-eq}. Blowing up, we can write, on a typical open
set in the blowup, 
\[x_{2i-1}=u_ix_1, \del_{x_{2i-1}}=(\del_{u_i}-u_i\del_{x_1})/x_1, i>1.
\]
Thus, the lifted bivector $\Pi'$ is holomorphic, as claimed. 
\par Next, we show that the lifted structure $\Pi'$ is P-normal, i.e.
that is has corank $>2$ on the singular locus of its Pfaffian divisor.
By the above computation, $\pf(\Pi')$ coincides with $\pi^*(\Pi)_{\mathrm{red}}$.
Moreover, 
at least in the open set considered above, 
this divisor is given by $x_1u_2...u_k$ and $\Pi'$ has corank $>2$ on the singular locus
of that divisor, i.e. intersections of pairs of smooth branches. By our assumption of weak P-normality,
the latter locus is dense in the entire singular locus of $\pf(\Pi')$, and it follows that $\Pi'$
has corank $>2$ on its entire singular locus, which proves P-normality for $\Pi'$.
\end{proof}
We will call $\tilde X$, together with the P-normal Poisson structure lifted from $X$,
the \emph{P-normal resolution} of the weakly P-normal Poisson manifold $(X, \Pi)$.
\section{Proof of Theorem}
We will assume to begin with that $(X, \Pi)$ is P-normal.
We will prove first that $(X, \Pi)$ is unobstructed.
\subsection{Step 1: from log differentials to log tangent vectors} 
From Deligne's 
$E_1$-degeneration theorem it follows that $d$ induces the 
zero
map on $\HH^1(\Omega^.\llog{D}), D=\pf(\Pi)$. Therefore by the Cartan formula it follows
in the usual manner  that
the bracket-induced map (not the wedge product map)
\[\HH^i(\Omega^._X\llog{D})\otimes \HH^j(\Omega^._X(\llog{D})
\to \HH^{i+j}(\Omega^._X(\llog{D})
\]
vanishes.
Therefore by the duality Lemma \ref{duality-iso-lem}, the same is true of the bracket-induced map

\[\HH^i(T^._X\mlog{D})\otimes \HH^j(T^._X(\mlog{D})
\to \HH^{i+j}(T^._X(\mlog{D})
\]
Consequently, all the obstruction maps vanish:
\[\sym^m\HH^1(T^._X\mlog{D})\to \sym^{m-2}\HH^1(T^._X\mlog{D})\otimes \HH^2(T^._X).\]
Now the $m$th order Poisson 
deformation space is computed by the Jacobi (or Quillen) complex
$J_m(T^._X\mlog{D})$
(\cite{cid} or \cite{atom}) .
By the above, the $E_1$ spectral sequence 

\[E_1^{p,q}=H^q({\bigwedge\limits^p}_\C T^._X\mlog{D})
\Rightarrow \HH^i(J_m(T^._X\mlog{D}))\]
degenerates at $E_1$: 
Therefore the log-Schouten dgla $T^._X\mlog{D}$ has unobstructed deformations.
\subsection{Step 2: from log to ordinary tangent vectors}
Heuristically at least, 
because $D$ is the Pfaffian divisor of $\Pi$, a Poisson deformation of $(X, \Pi)$
automatically extends to a deformation of $D$ as well. This implies that the
map on deformation theories induced by the
inclusion $i:T^._X\mlog{D}\to T^._X$ admits a right inverse.
Therefore $T^._X$ has unobstructed deformations, i.e. $(X, \Pi)$ is unobstructed
as Poisson manifold.\par
Formally speaking, consider the complex $K^.$ defined by
\eqspl{}{
K^0=T;\\
K^1=T^2\oplus N^0_D;\\
K^i=T^{i+1}\oplus T^{i-1}\otimes N^0_D, i\geq 2, 
}
$N^0_D\subset N_D=\O_D(D)$
being the 'locally trivial' subsheaf of the normal sheaf, image of $T_X\to N_D$;
the differentials are given by:
$T^i\to T^{i+1}$ is Poisson differential  $[.,\Pi]$,  and the map
$T^i\to T^{i-1}\otimes N^0_D$ is extended via tensor product from the canonical
surjective  map
$T\to N^0_D$ (with kernel $T\llog{D}$). 
The differential is by definition zero on the $T^{i-1}\otimes\O_D(D)$ summand.
Now there  is an obvious inclusion map
\[T^.\mlog {D}\to K^.\]
which is a quasi-isomorphism,
and a projection
\[K^.\to T^..\]
On the other hand, there is also a map
\[T^.\to K^.\]
where $T^i\to T^i$ is the identity and $T^i\to T^{i-2}$ is extended via tensor product
from the map
\eqspl{}{
T^2\to N^0_D\subset \O_D(-K_X),\\
u\mapsto nu\wedge \Pi^{n-1}|_D
} (the latter is just the derivative of the map
$\Pi\mapsto \Pi^n$). Note that this map indeed goes into the $N^0_D$
subsheaf due to the fact that $\Pi^{n-1}$ vanishes on $\sing(D)$, which follows from the
assumption of P-normality.
Thus in the derived category, $T^.$ is a direct summand  of $K^.\sim T^.\mlog{D}$,
 and since the latter has vanishing obstructions, so does the former.
  This completes the proof that $(X, \Pi)$ is unobstructed, i.e.
  assertion (i) of Theorem \ref{main-thm} . \par
 For assertion (ii),  the local triviality of the induced deformation of $D=\pf(\Pi)$
  follows from the fact that the map $T^2\to N_D$ above goes into $N^0_D$.
  In fact, deformations of $(X,D)$ generally are controlled by the complex
  \[T\to N_D,\]
  while the locally trivial deformations (i.e. those where $D$ deforms locally trivially)
  are controlled by the subcomplex
  \[(T\to N^0_D)\sim T\mlog{D}.\]
  In our case we get a map
  \[T^.\to (T\to N_D)\]
  which factors through $(T\to N^0_D)\simeq T\mlog{D}$. Indeed this map is just the composite
  of the map $T^.\to T^.\mlog{D}$ constructed above with the truncation map 
  $T^.\mlog{D}\to T\mlog{D}$.
  \par As for assertion (iii),
because locally trivial deformations of $(X, D)$ 
 are controlled by $T\mlog{D}$, the proof of their unobstructedness 
 is identical, replacing the complexes $ T^.\mlog{D}, \Omega^.\llog{D}$ 
  by their zeroth component (see the Appendix below for a more
  general statement). 
 Now consider the following diagram
 \eqspl{}{
 \begin{matrix}
 T^.\mlog{D}&\to&T\mlog{D}\\
 \downarrow&&\downarrow\\
 T^.&\to&T
 \end{matrix}
 }inducing
  \eqspl{}{
  \begin{matrix}
  \HH(T^.\mlog{D})&\to&H(T\mlog{D})\\
  \downarrow&&\downarrow\\
  \HH(T^.)&\to&H(T).
  \end{matrix}
  } Here the left vertical arrow is a direct summand projection.
The top horizonatal arrow
is surjective by Deligne's $E_1$ degeneration of the spectral sequence for $\HH(\Omega^.\llog{D})$.
Thus, the deformation space corresponding to $T^.\mlog{D}$ maps smoothly to
the locally trivial deformations of $(X, D)$, as well as to the Poisson deformations
of $(X, \Pi)$, proving assertion (iii).
   This completes the proof  of Theorem \ref{main-thm} in case $(X, \Pi)$ is P-normal.
   \subsection{Step 3: weakly P-normal case}
   It remains to
  consider case where $(X, \Pi)$ is only assumed weakly P-normal. 
  As in \S \ref{wp-sec}, let $\pi:(X', \Pi')\to (X, \Pi)$ be the P-normal resolution, i.e. the
   iterated blowup of $(X, \Pi)$ along its Pfaffian stratification
  \[\pf^n(\Pi)\subset \pf^{n-1}(\Pi)\subset...\subset\pf(\Pi).\]
  Then by Proposition \ref{lifting-prop}, 
  $\Pi'$ is the holomorphic lift of $\Pi$ and $\Pi'$ is P-normal. Now a deformation 
  of $(X, \Pi)$ induces a locally trivial deformation of $\pf(\Pi)$, hence a deformation of
  the above Pfaffian stratification, hence a deformation of the blowup $(X', \Pi')$.
  Conversely, a deformation of $(X', \Pi')$ induces a locally trivial deformation
  of $\pf(\Pi')$ and hence of the individual exceptional divisors of $\pi$, which can then
  be blown down to yield a deformation of $(X, \Pi)$. Thus, the 
  formal deformation functors of $(X, \Pi)$ and
  $(X', \Pi')$ are equivalent, and since $(X', \Pi')$ has unobstructed deformations, so does 
  $(X, \Pi)$.


%
%
\section{An obstruction}
It turns out that when $\Pi$ is P-normal, the singularities of $D$ are not 'accidental' and in fact are locally constant
over the Poisson moduli. This leads one to suspect the singular strata of a normally-crossing
$D$ may represent some kind of characteristic class
associated to $\Pi$, even though  these singular strata 
fail to have the generic codimension in terms of the rank
of the tensor $\Pi$ considered as skew-symmetric bundle map. One result in this direction is given
in Proposition \ref{obstruction-prop} below, which gives a Chern-polynomial
obstruction to the smoothness (even in codimension 3) of $D$ (i.e. when the polynomial
is nonzero, the singular locus of $D$ has codimension 3 in $D$ (codimension 4 in $X$)).
This result is actually a special case of a theorem of Gualtieri and Pym \cite{gualtieri-pym},
but as the proof is short and elementary, we include it.
\par
First a definition. We say that a complex manifold $X$ is \emph{compact through codimension $c$}
if it is of the form $X=Y\setminus Z$ with $Y$ a compact complex manifold 
and $Z$ an analytic subset of codimension
$>c$. For example, if $E$ is a vector bundle on $Y$ carrying a skew-symmetric bundle map
$\phi$ that degenerates in the generic codimension, then the open set where $\phi$
has corank $\leq 2$ is compact through codimension 5.
\begin{prop}\label{obstruction-prop}
Let $X$ be a complex manifold compact through codimension 4, and suppose $X$ carries a
generically nondegenerate Poisson structure $\Pi$ 
with smooth Pfaffian divisor $D$ on which $\Pi$ has corank exactly 2. Then, where $c_i=c_i(X)$,
we have
\eqspl{obstruction-eq}{c_1(c_1c_2-c_3)=0.}
\end{prop}
\begin{proof}
As $D$ is an anticanonical divisor on $X$, it is Calabi-Yau, i.e. $c_1(\Omega_D)=0$.
Moreover, if $M$ denotes the kernel of $\Pi^\#$ on $\Omega_X|_D$, then by assumption $M$ is
a rank-2 bundle and by Lemma \ref{conormal-in-kernel-lem}, $M $ has the conormal
line bundle  $\check N$ of $D$ as subbundle.
Therefore $M/\check N$ is a line subbundle of $\Omega_D$, such that
$\Omega_D/(M/\check N)$ carries a nondegenerate alternating form given by
the restriction of $\Pi$, hence is self-dual and has vanishing odd Chern classes. 
Because $c_1(D)=0$ as well, it follows that
\[c_{\mathrm{odd}}(D)=0.\]
Then we get the result via Gysin by calculating $c_3(D)$ from
\[c(\Omega_D)=\frac{c(\Omega_X)}{1+c_1(\Omega_X)}.\]
\end{proof}
\begin{remarks}
(i) The generic codimension where the skew-symmetric map $\Pi^\#$ has corank $>2$
(i.e. $\geq 4$) is 
$6=\binom{4}{2}$ (see \cite{harris-tu} or \cite{ful}, 14.4.11),
so the proposition gives an obstruction for a manifold $X$ to support a Poisson structure
with 'generic singularities',
specifically a dimensionally proper 'second' (corank $\geq 4$) degeneracy locus
$S$ and a Pfaffian smooth outside $S$.
\par (ii) Note that the condition \eqref{obstruction-eq} fails for $\P^4$, so $\P^4$ has no
'nice' Poisson structure in the foregoing sense ($D$ smooth), though it might still have
one with $D$ normally crossing. There is, in fact, a family of
pseudo-symplectic Poisson
structures on $\P^4$ due to Feigin-Odesskii \cite{feigin-odesskii},
whose Pfaffian divisor is singular on an elliptic curve. It seems likely,
but we haven't checked, that these Poisson structure lifts to P-normal ones
on the blowup of $\P^4$ along the elliptic curve. If so, it would presumably
follow, as in the weakly P-normal case, that these structures
are unobstructed.\par
(iii) Gualtieri and Pym \cite{gualtieri-pym} have proven, assuming
only that $D$ is reduced, that the singular locus
 $\sing(D)$ has codimension at most 3 in $X$, and if the codimension
equals 3, then the fundamental class of $\sing(D)$
is $c_1c_2-c_3$. It would be interesting to find a formula for the
fundamental class of $\sing(D)$ when $D$ is P-normal, so that
the singular locus has codimension 2 in $X$.
\end{remarks}
\section{Appendix : manifolds with a normal-crossing anticanonical divisor}\label{appendix}
The foregoing methods yield a short  proof of the following result,
obtained earlier at least in the case $D$ smooth by Iacono \cite{iacono}
 (no Poisson structure is assumed).\footnote{We thank D. Iacono
for pointing out that 'locally trivial' is required in the statement}
\begin{thm} Let $X$ be compact K\"ahler with 
an effective locally normal-crossing anticanonical divisor $D$.
Then $(X,D)$ has unobstructed locally trivial deformations.\end{thm}
The assertion amounts to saying that the dgla $T_X\llog{D}$ has unobstructed deformations.
The isomorophism \[T_X\simeq \Omega_X^{n-1}(-K_X)\simeq \Omega_X^{n-1}(D),n=\dim(X),\]
induces
\[T_X\mlog{D}\simeq \Omega^{n-1}_X\llog{D}\]
Note that we have the  interior multiplication pairing
\[T_X\mlog{D}\times\Omega_X^i\llog{D}\to\Omega^{i-1}_X\llog{D}\]
and the Cartan formula for Lie derivative
\[L_v(\omega)=i_vd\omega\pm d(i_v(\omega)).\]
Deligne's $E_1$-degeneration shows that $d$ induces the zero map on $H^.(\Omega^._X\llog{D})$,
and it follows that the pairing induced by Lie derivative vanishes:
\[H^1(T_X\mlog{D})\times H^1(\Omega_X^{n-1}\llog{D})\to H^2(\Omega^{n-1}_X\llog{D}).\]
By the above identification, this can be identified with the pairing on $H^1(T_X\mlog{D})$
induced by bracket, i.e. the obstruction pairing. Therefore the obstruction vanishes, so
$(X,D)$ has unobstructed deformations.\qed\par
It is not in general true that $|-K_X|$ is constant under deformations on $X$, even if it has a smooth 
member, as shown by the example of $\P^2$ blown up in 10 points on a smooth cubic.
Therefore, $\Def(X,D)$ is in general a smooth fibration
over a smooth proper subvariety of  $\Def(X)$.
\section{Appendix: a computation}\label{computation} Here we will compute the Poisson bivector
$\Pi\sbr 2.$ seen in \S \ref{hilb}, Display \eqref{2Pi},
 where $\Pi$ is a Poisson bivector on $S$ locally
of the form $x\del_x\wedge\del_y$. Henceforth, we will omit the wedge product
symbols and write $\del x$ for $\del/\del x$ etc. We can write
\eqspl{}{ 2\Pi\sbr 2.=\Pi_1+\Pi_2,
\Pi_1=(x_1\del x_1+\del x_2)(\del y_1+\del y_2)=:t^1_xt^1_y,\\
\Pi_2=\left ((y_1-y_2)(x_1\del x_1-x_2\del x_2)\right )\cdot \frac{\del y_1-\del y_2}{y_1-y_2}
=:t^2_x\cdot t^2_y.
}
Now the rational vector fields $t^i_x, t^i_y$ on $S\sbr 2.$ can be expressed in terms
of the frame $\del a_0, \del a_1, \del b_0, \del b_1$ by evaluating on the
$a_i, b_j$, which yields, on setting $\delta_y=(y_1-y_2)^2=b_1^2-4b_2$,
\eqspl{}{
t^1_x=a_0\del a_0+a_1\del a_1, t^1_y=-(b_1\del b_0+\del b_1),\\
t^2_x=(-2b_0a_1+b_1a_0)\del a_0+(2a_0-a_1b_1)\del a_1,\\
t^2_y=-\del b_0+\frac{a_1b_1}{\delta_y}\del a_0-2\frac{a_1}{\delta_y}\del a_1.
}
Consequently,
\eqspl{}{
\Pi_2=(2b_0a_1-b_1a_0)\del a_0\del b_0-a_1^2\del a_0\del a_1-(2a_0-a_1b_1)\del a_1\del b_0.
}
Then a routine calculation yields \eqref{2Pi}.

\ds\ds\ds
\bibliographystyle{amsplain}
\bibliography{../mybib}
\end{document}